\documentclass[11pt,reqno]{amsart}
\usepackage{amsmath,amssymb,amsfonts,amstext,amscd,latexsym,amsthm,mathrsfs}
\usepackage{mathtools}
\usepackage{graphicx}
\usepackage{hyperref}
\usepackage{a4wide}
\usepackage{enumerate}
\newcommand{\R}{\ensuremath{\mathbb{R}}}

\newcommand{\tr}{\ensuremath{\mathrm{tr}}\,} %trace
\newcommand{\pd}{\partial}
\newcommand{\cd}{\nabla}
\newcommand{\grad}{\mathrm{grad}\,} %gradient
\newcommand{\E}{{\rm e}} %base of the natural logarithm
\newcommand{\inner}[2]{\left\langle #1,#2\right\rangle} %Euclidean inner product
\newcommand{\norm}[1]{\left|\left|#1\right|\right|} %Euclidean norm
\newcommand{\X}{{X}} %flow parametrisation 
\newcommand{\M}{{M}} %flow parametrisation 
\newcommand{\A}{\W} %second fundamental form 
\newcommand{\W}{\mathcal{W}} %Weingarten map
\newcommand{\eL}{\mathscr{L}} %Linearisation operator

\newcommand{\lb}{\left(}
\newcommand{\rb}{\right)}
\newcommand{\lsb}{\left[}
\newcommand{\rsb}{\right]}

\newcommand{\F}{F} %curvature function
\newcommand{\xF}{{\F_x}}
\newcommand{\xA}{{\A^x}}
\newcommand{\xW}{{\W^x}}
\newcommand{\xnu}{{\nu_x}}
\newcommand{\yF}{{\F_y}}
\newcommand{\yA}{{\A^y}}
\newcommand{\yW}{{\W^y}}
\newcommand{\ynu}{{\nu_y}}
\theoremstyle{plain}
\numberwithin{equation}{section}
\newtheorem{thm}{Theorem}[section]

\newtheorem{lem}[thm]{Lemma}

\newtheorem{prop}[thm]{Proposition}
\newtheorem{defn}[thm]{Definition}

\newtheorem*{conds*}{Conditions}
\newtheorem*{auxconds*}{Ancillary Conditions}
\newtheorem*{props*}{Properties}
\theoremstyle{remark}
\newtheorem*{rem}{Remark}

\def\be #1\ee {\begin{eqnarray} #1\end{eqnarray}}
\def\benn #1\eenn {\begin{eqnarray*} #1\end{eqnarray*}}
\def\ba #1\ea {\begin{align} #1\end{align}}
\def\bann #1\eann {\begin{align*} #1\end{align*}}
\def\ben #1\een {\begin{enumerate} #1\end{enumerate}}
\def\bi #1\ei {\begin{itemize} #1\end{itemize}}
\title[Two-sided non-collapsing curvature flows]{Two-sided non-collapsing curvature flows} %in the Simply Connected Space Forms}

\author{Ben Andrews}
\address{Mathematical Sciences Institute, Australian National University, ACT 0200 Australia}
\address{Mathematical Sciences Center, Tsinghua University, Beijing 100084, China}
\email{ben.andrews@anu.edu.au}
\thanks{Research partially supported by Discovery grant DP120100097 of the Australian Research Council.}
\author{Mat Langford}
\address{Mathematical Sciences Institute, Australian National University, ACT 0200 Australia}
\email{mathew.langford@anu.edu.au}\thanks{The second author gratefully acknowledges the support of an Australian Postgraduate Award and an Australian National University HDR Supplementary Scholarship, and the support and hospitality of the Mathematical Sciences Center at Tsinghua University, Beijing, and the Department of Mathematics at East China Normal University, Shanghai during the completion of part of this work.}

\begin{document}

\begin{abstract}
It was recently shown that embedded solutions of curvature flows in Euclidean space with concave (convex), degree one homogeneous speeds are interior (exterior) non-collapsing \cite{NC}. These results were subsequently extended to hypersurface flows in the sphere and hyperbolic space \cite{NC2}. In the first part of the paper, we show that locally convex solutions are exterior non-collapsing for a larger class of speed functions than previously considered; more precisely, we show that the previous results hold when convexity of the speed function is relaxed to inverse-concavity. We note that inverse-concavity is satisfied by a large class of concave speed functions \cite{An07}. As a consequence, we obtain a large class of two-sided non-collapsing flows, whereas previously two-sided non-collapsing was only known for the mean curvature flow.

In the second part of the paper, we demonstrate the utility of two sided non-collapsing with a straightforward proof of convergence of compact, convex hypersurfaces to round points.

The proof of the non-collapsing estimate is similar to the previous results mentioned, in that we show that the exterior ball curvature is a viscosity supersolution of the linearised flow equation. The new ingredient is the following observation: Since the function which provides an upper support in the derivation of the viscosity inequality is defined on $M\times M$ (or $TM$ in the `boundary case'), whereas the exterior ball curvature and the linearised flow equation depend only on the first factor, we are privileged with a freedom of choice in which second derivatives from the extra directions to include in the calculation. The optimal choice is closely related to the class of inverse-concave speed functions.
\end{abstract}

\maketitle

\section{Introduction}

We consider embedded solutions $\X:\M^n\times [0,T)\to N_\sigma^{n+1}$ of curvature flows of the form
\begin{equation}\label{eq:CF}\tag{CF}
\pd_t\X(x,t)=-F(x,t)\nu(x,t)\,,
\end{equation}
where $N_\sigma^{n+1}$ is the complete, simply connected Riemannian manifold of constant curvature $\sigma\in \{-1,0,1\}$ (that is, either hyperbolic space $H^{n+1}$, Euclidean space $\R^{n+1}$, or the sphere $S^{n+1}$), $\nu$ is a choice of normal field %\footnote{When the ambient space is $\R^{n+1}$ or $H^{n+1}$, we choose $\nu$ to be the outer normal.} 
for the evolving hypersurface $\X$, and the speed $F$ is given by a smooth, symmetric, degree one homogeneous function of the principal curvatures $\kappa_i$ of $\X$ which is monotone increasing with respect to each $\kappa_i$. Equivalently, $F$ is a smooth, monotone increasing, degree one homogeneous function of the Weingarten map $\W$ of $\X$. Moreover, we will always assume that $F$ is normalised such that $F(1,\dots,1)=1$; however, this is merely a matter of convenience--all of the results hold, up to a recalibration of constants, in the un-normalised case.

The \emph{interior} and \emph{exterior ball curvatures} \cite{NC} of a family of embeddings $\X:\M^n\times [0,T)\to \R^{n+1}$ with normal $\nu$ are, respectively, defined by $\overline k(x,t):=\sup_{y\neq x}k(x,y,t)$, and $\underline k(x,t):=\inf_{y\neq x}k(x,y,t)$, where
\ba\label{eq:k}
k(x,y,t):=\frac{2\inner{\X(x,t)-\X(y,t)}{\nu(x,t)}}{\norm{\X(x,t)-\X(y,t)}^2}\,.
\ea
Equivalently, $\overline k(x,t)$ (resp. $\underline k(x,t)$) gives the curvature of the largest region in $\R^{n+1}$ with totally umbilic boundary that lies on the opposite (resp. same) side of the hypersurface $\X(\M,t)$ as $\nu(x,t)$, and touches it at $\X(x,t)$ (with sign determined by $\nu$) \cite[Proposition 4]{NC}. Therefore, for a compact, convex embedding with outward normal, they are, respectively, the curvature of the largest enclosed, and smallest enclosing spheres which touch the embedding at $\X(x,t)$. It follows that $\kappa_{\max}\leq\overline k$ and $\kappa_{\min}\geq \underline k$.

For flows in Euclidean space, embedded, $F>0$ solutions of \eqref{eq:CF} are \emph{interior non-collapsing} when the speed is a concave function of the Weingarten map, and \emph{exterior non-collapsing} when the speed is a convex function of the Weingarten map \cite{NC}; more precisely, there exist $k_0\in\R$, $K_0>0$ such that $\underline k\geq k_0 F$ in the former case, and $\overline k\leq K_0 F$ in the latter. In particular, solutions of mean convex mean curvature flow (in which case the speed is the trace of the Weingarten map) are both interior and exterior non-collapsing. Two-sided non-collapsing has many useful consequences; for example, for uniformly convex hypersurfaces we obtain uniform pointwise bounds on the ratios of principal curvatures, and a uniform bound on the ratio of circumradius to in-radius. This leads to a new proof of Huisken's theorem \cite{Hu84} on the convergence of convex hypersurfaces to round points. Moreover, one-sided non-collapsing can also provide useful information; for example, interior non-collapsing rules out certain singularity models, such as products of the Grim Reaper curve with $\R^{n-1}$. For convex speeds, exterior non-collapsing is sufficient to obtain a bound on the ratio of circumradius to in-radius, and the proof of convergence of locally convex initial hypersurfaces to round points \cite{An94}, is also simplified. %We will develop the applications further in the upcoming paper \cite{}.

More recently, it was shown that the above statements have natural analogues when the ambient space is either the sphere or hyperbolic space \cite{NC2}: Considering the sphere $S^{n+1}$ as the embedded submanifold $\{X\in \R^{n+2}: \inner{X}{X}=1\}$ of $\R^{n+2}$, and hyperbolic space $H^{n+1}$ as the embedded submanifold $\{X\in \R^{n+1,1}: \inner{X}{X}=-1\}$ of Minkowski space $\R^{n+1,1}$, the function $k$ may be formally defined as in \eqref{eq:k}, except that now we take $\inner{\,\cdot\,}{\,\cdot\,}$ and $\norm{\,\cdot\,}$ to be the inner product and induced norm on, in the case of the sphere, $\R^{n+2}$, and, in the case of hyperbolic space, the spacelike vectors in $\R^{n+1,1}$. Then, if $F$ is a concave function of the curvatures, there exists $K_0>0$ such that
\bann
\frac{\overline k}{F}-\frac{1}{\eta}\leq K_0 \E^{-2\sigma \eta t}\,,
\eann
and, if $F$ is a convex function of the curvatures, there exists $k_0\in \R$ such that
\bann
\frac{\underline k}{F}-\frac{1}{\eta}\geq k_0 \E^{-2\sigma \eta t}\,,
\eann
where $\eta>0$ depends on bounds for the derivative of $F$.

Let us recall the following definition \cite{An07}:
\begin{defn}%[Inverse-concavity]
Let $\Omega$ be an open subset of the positive cone $\Gamma_+:=\{z\in \R^n\,:\,z_i>0\;\mathrm{for\; all}\; i\}$. Define the set $\Omega_{\ast}:=\{(z_1^{-1},\dots,z_n^{-1})\}$. Then a function $f:\Omega\to\R$ is called \emph{inverse-concave} if %$\Omega_\ast\subset\Omega$ and 
the function $f_\ast:\Omega_\ast\to\R$ defined by $f_\ast\lb z_1^{-1},\dots,z_n^{-1}\rb:=f\lb z_1,\dots,z_n\rb^{-1}$ is concave. %If the speed function $F$ is given by $F=f(\kappa_1,\dots,\kappa_n)$, such that $f$ is inverse-concave then we say that $F$ is \emph{inverse-concave}.
\end{defn}

The main result of this article may now be stated as follows:
\begin{thm}\label{thm:NC}
Let $f:\Gamma_+\to\R$ be a smooth, symmetric function which is homogeneous of degree one, and monotone increasing in each argument. Let $\X$ be a solution of \eqref{eq:CF} with speed given by $F=f(\kappa_1,\dots,\kappa_n)$, where $\kappa_i$ are the principal curvatures of $\X$. Then, if $f$ is inverse-concave, $\X$ is \emph{exterior non-collapsing}; that is,
\ben
\item\label{case:R} If $N^{n+1}=\R^{n+1}$, then, for all $(x,t)\in\M\times[0,T)$,
\bann
\frac{\underline k(x,t)}{F(x,t)} \geq \inf_{M\times\{0\}}\frac{\underline k}{F}\,.
\eann
\item\label{case:S} If $N^{n+1}=S^{n+1}$, and $\tr(\dot F)\leq\eta$, then, for all $(x,t)\in\M\times[0,T)$,
\bann
\frac{\underline k(x,t)}{F(x,t)}-\frac{1}{\eta}\geq \inf_{M\times\{0\}}\left(\frac{\underline k}{F}-\frac{1}{\eta}\right) \E^{-2\eta t}\,.
\eann
\item\label{case:H} If $N^{n+1}=H^{n+1}$, and $\tr(\dot F)\geq \eta$, then, for all $(x,t)\in\M\times[0,T)$,
\bann
\frac{\underline k(x,t)}{F(x,t)}-\frac{1}{\eta }\geq \inf_{M\times\{0\}}\left(\frac{\underline k}{F}-\frac{1}{\eta}\right) \E^{2\eta t}\,,
\eann
where $\dot F$ is the derivative of $F$ with respect to the Weingarten map.
\een
\end{thm}
In particular, combining Theorem \ref{thm:NC} with the previous non-collapsing results \cite{NC,NC2}, we find that solutions of flows in spaceforms by concave, inverse-concave speed functions are both interior and exterior non-collapsing. We note that concave speed functions satisfy $\tr(\dot F)\geq 1$, so in that case we may take $\eta=1$ in case (\ref{case:H}) of Theorem \ref{thm:NC}.

We note that the class of admissible speeds which are both concave and inverse-concave is surprisingly large \cite{An07}, and includes, for example, the degree one homogeneous ratios and roots of the elementary symmetric polynomials.

The authors wish to express their thanks to Chen Xuzhong and Yong Wei for their helpful comments and suggestions on earlier versions of this work.

\section{Proof of Theorem \ref{thm:NC}}

We first extend (cf. \cite{NC}) $k(\,\cdot\,,\,\cdot\,,t)$ to a continuous function on the compact manifold with boundary $\widehat \M$. As a set, $\widehat \M:=(\M\times \M\setminus D)\sqcup S\M$, where $D:=\{(x,x):x\in \M\}$ is the diagonal submanifold and $SM$ is the unit tangent bundle with respect to the metric at time $t$. The manifold-with-boundary structure is defined by the atlas generated by all charts for $(M\times M)\setminus D$, together with the charts $\widehat Y$ defined by $\widehat Y(z,s):=\big(\exp(sY(z)),\exp(-sY(z))\big)$ for $s$ sufficiently small, where $Y$ is a chart for $SM$. The extension is then given by setting $k(x,y,t):=\W_{(x,t)}(y,y)$ for $(x,y)\in S_{(x,t)}\M$. 

We also recall some useful notation \cite{NC}; namely, we define
\bann
d(x,y,t):=\norm{\X(x,t)-\X(y,t)} \quad\mbox{and}\quad w(x,y,t):=\frac{\X(x,t)-\X(y,t)}{\norm{\X(x,t)-\X(y,t)}}\,,
\eann
and use scripts $x$ and $y$ to denote quantities pulled back to $\M\times \M$ by the respective projections onto the first and second factor. With this notation in place, $k$ may be written as
\bann
k=\frac{2}{d^2}\inner{dw}{\nu_x}\,.
\eann

%The main observation of this article is the following: Since $k(x,y,t)$ only depends on the point $x$ at the level of highest derivatives, we are able to make an arbitrary choice in the `$y-y$' components of the elliptic operator. Using this observation, we will prove that $\underline k$ is a viscosity supersolution of the linearised flow

Theorem \ref{thm:NC} is a direct consequence of the following proposition:

\begin{prop}\label{prop:visc}
If the flow speed $F$ is inverse-concave, then the exterior ball curvature $\underline k$ is a viscosity supersolution of the equation
\begin{equation}\label{eq:LF}
\pd_t u=\eL u%+\inner{\cd u}{R_u(\cd u)}_F
+\lb |\W|^2_F-\sigma \tr(\dot F)\rb u+2\sigma F\,,
\end{equation}
where $\eL:=\dot F^{ij}\cd_i\cd_j$, $\inner{u}{v}_F:=\dot F^{ij}v_iv_j$, and $|\W|^2_F:=\dot F^{ij}\W^2_{ij}$.%, and $R_\phi:=(\W-\phi I)^{-1}$.
\end{prop}

We note that the speed function satisfies the equation
\bann
\pd_t F=\eL F+\lb |\W|^2_F-\sigma \tr(\dot F)\rb F%+2\sigma F
\eann
under the flow \cite{An94b}.

\begin{proof}[Proof of Proposition \ref{prop:visc}]
Consider, for an arbitrary point $(x_0,t_0)\in\M\times [0,T)$, an arbitrary lower support funtion $\phi$ for $\underline k$ at $(x_0,t_0)$; that is, $\phi$ is $C^{2,1}$ on a backwards parabolic neighbourhood $P:=U_{x_0}\times (t_0-\varepsilon,t_0]$ of $(x_0,t_0)$, and $\phi\leq \underline k$ with equality at $(x_{0},t_{0})$. Then we need to prove that the differential inequality
\bann
\pd_t \phi\geq\eL \phi+\lb |\W|^2_F-\sigma \tr(\dot F)\rb \phi+2\sigma F
\eann
holds at $(x_{0},t_{0})$.  

We note that $k(x,y,t)\geq \underline k(x,t)\geq \phi(x,t)$ for all $(x,y,t)\in \widehat \M\times [0,T)$ such that $(x,t)\in P$, and, since $k$ is continuous and $\widehat M$ is compact, we either have $\underline k(x_{0},t_{0})=k(x_{0},y_{0},t_{0})$ for some $y_{0}\in \M\setminus\{x_{0}\}$, or $\underline k(x_{0},t_{0})=\A_{(x_0,t_0)}(y_0,y_0)$ for some $y_{0}\in S_{(x_{0},t_{0})}\M$. We consider the former case first.\\

\noindent\textbf{The interior case.}\\

We first suppose that $\inf_{\M}k(x_0,\,\cdot\,,t_0)<k(x_0,y_0,t_0)$ for all boundary points $(x_0,y_0)$ of $\widehat \M$. In that case, we have $\kappa_1(x_0,t_0)>\underline k(x_0,t_0)=k(x_{0},y_{0},t_{0})$ for some $y_{0}\in \M\setminus \{x_{0}\}$, and $k(x,y,t)\geq\underline k(x,t) \geq \phi(x,t)$ for all $(x,t)\in P$ and all $y\in \M\setminus\{x\}$. In particular, we have the inequalities
\ba\label{eq:lowersupport}
\begin{split}
\pd_t(k-\phi)\leq{}& 0\,,\\
%(\pd_{x_i}+{\Lambda_i}^p\pd_{y^p})(k-\phi)={}&0\,,\\
\mbox{and}\quad \widehat \eL(k-\phi)\geq{}& 0
\end{split}
\ea
at $(x_{0},y_{0},t_{0})$ for any elliptic operator $\widehat\eL$ on $\M\times\M$. We would like $\widehat\eL$ to project to $\eL$ on the first factor. This leads us to consider operators of the form $\widehat\eL=\dot F^{ij}_x\cd_{\pd_{x_i}+{\Lambda_i}^p\pd_{y^p}}\cd_{\pd_{x^j}+{\Lambda_j}^q\pd_{y^q}}$, where $\Lambda$ is any $n\times n$ matrix.

We note that, in both of the cases $\sigma=\pm 1$, the ambient Euclidean/Minkowskian derivative decomposes into tangential and normal components as $D=\overline D-\overline g\otimes\overline \nu$, where $\overline D$, $\overline g$, and $\overline \nu$ are, respectively, the induced connection, metric, and outer/future-pointing normal of $N_\sigma^{n+1}$ with respect to its embedding. Using the fact that $\inner{\overline \nu}{\overline\nu}=\sigma$, and that the ambient position vector is normal to $N_\sigma^{n+1}$, a straightforward computation yields
\ba\label{eq:Dk}
(\pd_{x^i}+{\Lambda_i}^p\pd_{y^p})k={}&\frac{2}{d^2}\left(\inner{\pd^x_i-{\Lambda_i}^p\pd^y_p}{\nu_x-kdw}+\inner{dw}{{\xA_i}^p\pd_p^x}\right)\,.
\ea
If we choose the co\"ordinates $\{x^i\}_{i=1}^n$ and $\{y^i\}_{i=1}^n$ to be orthonormal co\"ordinates (with respect to the induced metric $g$ at time ${t_0}$) centred at $x_0$ and $y_0$ respectively, then a further straightforward computation using the vanishing of \eqref{eq:Dk} at $(x_0,y_0,t_0)$ and the Codazzi equation yields
\ba
\cd_{\pd_{x^j}+{\Lambda_j}^q\pd_{y^q}}\cd_{\pd_{x^i}+{\Lambda_i}^p\pd_{y^p}}k={}&\frac{2}{d^2}\Big\{\inner{-\xA _{ij}\xnu-\sigma \delta_{ij}\X_x+{\Lambda_i}^p{\Lambda_j}^q\lb\yA_{pq}\ynu+\sigma \delta_{pq}\X_y\rb}{\nu_x-kdw}\nonumber\\
&+\inner{\pd^x_i-{\Lambda_i}^p\pd^y_p}{{\xA_j}^q\pd^x_q}-(\pd_{x^j}+{\Lambda_j}^q\pd_{y^q})k\inner{\pd^x_i-{\Lambda_i}^p\pd^y_p}{dw}\nonumber\\
&-k\inner{\pd^x_i-{\Lambda_i}^p\pd^y_p}{\pd^x_j-{\Lambda_j}^q\pd^y_q}+\inner{\pd^x_j-{\Lambda_j}^q\pd^y_q}{{\xA_i}^p\pd^x_p}\nonumber\\
&+\inner{dw}{\cd \xA_{ij}-\sigma\xA_{ij}\X_x-{\xA_i}^r\xA_{rj}\nu_x}\nonumber\\
&-(\pd_{x^i}+{\Lambda_i}^p\pd_{y^p})k\inner{\pd^x_j -{\Lambda_j}^q\pd^y_q}{dw}\Big\}\,\label{eq:D2k}
\ea
at the point $(x_0,y_0,t_0)$. %Note that the Codazzi identity gives
%\bann
%\cd_i\W_{jk}-\cd_j\W_{ik}=\bar R_{ijk\alpha}\nu^\alpha
%\eann

Next, noting that the normal satisfies $D_t\nu= \overline D_t\nu+\sigma F\X=\grad F+\sigma F\X$, we compute
\ba\label{eq:Dtk}
\pd_t k={}&\frac{2}{d^2}\big(\inner{-\xF\nu_x+\yF\nu_y}{\nu_x-kdw}+\inner{dw}{\grad \xF+\sigma F \X_x}\big)\,.
\ea
%We note that all three equations hold for each of the three choices of ambient space. 
%Now we define an elliptic operator $\widehat\eL$ on $\M\times\M$ by $\dot F_x^{ij}\cd_{\pd_{x^i}+{\Lambda_i}^p\pd_{y^p}}\cd_{\pd_{x^j}+{\Lambda_j}^q\pd_{y^q}}$. 

Combining \eqref{eq:D2k} and \eqref{eq:Dtk}, we obtain
\bann
\left(\pd_t-{\widehat\eL}\,\right)k={}& \frac{2}{d^2}\Big\{\inner{\yF\nu_y-\dot F_x^{ij}{\Lambda_i}^p{\Lambda_j}^q\lb{\yA}_{pq}\ynu+\sigma \delta_{pq}\X_y\rb}{\nu_x-kdw}\\
&+k\dot F^{ij}_x\inner{\pd^x_i-{\Lambda_i}^p\pd^y_p}{\pd^x_j-{\Lambda_j}^q\pd^y_q} -2\dot F^{ij}_x\inner{\pd^x_j-{\Lambda_j}^q\pd^y_q}{{\xA_i}^p\pd^x_p}\\
&+\sigma \tr(\dot F_x)\inner{\X_x}{\nu_x-kdw}+2\sigma \xF\inner{\X_x}{dw}\Big\}\\
&+\frac{4}{d^2}\dot F^{ij}_x\cd_{\pd_{x^i}+{\Lambda_i}^p\pd_{y^p}}k\inner{\pd^x_j -{\Lambda_j}^q\pd^y_q}{dw}+|\xA|^2_F k
\eann
at the point $(x_0,y_0,t_0)$.

We now note that the vanishing of the $y$-derivatives at an off-diagonal extremum $y_0\in \M$ of $k(x_0,\cdot,t_0)$ determines the tangent plane to $X$ at $y_0$:
%\begin{lem}[\cite{NC}]\label{lem:ytangent}
%Suppose that the point $(x_{0},y_{0},t_0)$ is an off-diagonal extremum of $k$; that is, $y_0\neq x_0$ is an extremum of $k(x_0,\,\cdot\,,t_0)$. Then, 
%$$
%\ynu=\xnu-dkw
%$$
%at $(x_0,y_0,t_0)$.
%\end{lem}
%\begin{proof}[Proof of Lemma \ref{lem:ytangent}]
%Since $k(x_0,y_0,t_0)=\underline k(x_0,t_0)$, there is an interior ball $B$ of radius $1/k(x_0,y_0,t_0)$ touching the hypersurface to first order at the points $\X(x_0,t_0)$ and $\X(y_0,t_0)$ \cite[Proposition 4]{NC}. The outward normals to $B$ at these points agree with the outward normals to the hypersurface. In particular,
%\bann
%\nu(y_0,t_0) ={}& k(x_0,y_0,t_0)\lb\X(y_0,t_0)-\lb\X(x_0,t_0)- \frac{1}{k(x_0,y_0,t_0)}\nu(x_0,t_0)\rb\rb \\
%={}&(\xnu-kdw)|_{(x_0,y_0,t_0)}\,.
%\eann
%\end{proof}
%Since in addition
%\bann
%\inner{\X_y}{\pd^x_i-2\inner{\pd^x_i}{w}w}=0
%\eann
%at $(x_0,y_0,t_0)$, we see that span$\{\pd^x_i-2\inner{\pd^x_i}{w}w\}_{i=1}^n=$ span$\{\pd_i^y\}_{i=1}^n$ at $(x_0,y_0,t_0)$. %In particular, the tangent plane at $y_0$ is the reflection of the tangent plane at $x_0$ about the hyperplane normal to $w$.

\begin{lem}[\cite{NC,NC2}]\label{lem:ytangent}
Suppose that a point $(x,y,t)$ is an off-diagonal extremum of $k$; that is, $y\neq x$ is an extremum of $k(x,\,\cdot\,,t)$. Then
$$
\mathrm{span}\{\pd^x_i-2\inner{\pd^x_i}{w}w\}_{i=1}^n=\mathrm{span}\{\pd_i^y\}_{i=1}^n
$$
at $(x,y,t)$, where $\{\pd^x_i\}^n_{i=1}$ and $\{\pd^y_i\}$ are bases for $T_xM$ and $T_yM$ respectively.
\end{lem}
\begin{proof}[Proof of Lemma \ref{lem:ytangent}]
We may assume that $\{\pd_i^x\}_{i=1}^n$ and $\{\pd^y_i\}_{i=1}^n$ are orthonormal. Then $\{\pd^x_i-2\inner{\pd^x_i}{w}w\}_{i=1}^n$ is also orthonormal; note also that $\norm{\nu_x-kdw}=1$. Next, observe that the vanishing of $\pd_{y^i}k$ implies
\bann
\inner{\pd_i^y}{\nu_x-kdw}=0
\eann
for each $i$. If $\sigma\neq 0$, a short computation, using $d^2=2(\sigma-\inner{\X_x}{\X_y})$, yields
\bann
\inner{\X_y}{\nu_x-kdw}=0\,.
\eann
Thus, the orthogonal compliment of $\mathrm{span}\{\pd^y_i\}_{i=1}^n$ is $\mathrm{span}\{\sigma\X_y,\nu_x-kdw\}$. On the other hand, one easily computes
\bann
\inner{\pd^x_i-2\inner{\pd^x_i}{w}w}{\nu_x-kdw}=0
\eann
for each $i$, and, for $\sigma\neq 0$,
\bann
\inner{\pd^x_i-2\inner{\pd^x_i}{w}w}{\X_y}=0\,.
\eann
Thus, $\mathrm{span}\{\pd^x_i-2\inner{\pd^x_i}{w}w\}^\perp=\mathrm{span}\{\sigma\X_y,\nu_x-kdw\}$. The claim follows.
\end{proof}

Thus, without loss of generality, we may assume
\bann
\pd_i^y=\pd_i^x-2\inner{\pd_i^x}{w}w\,
\eann
at $(x_0,y_0,t_0)$. Note also that, when $\sigma\neq 0$,
%\bann
%\inner{\X_x}{dw}=\inner{\X_x}{\X_x-\X_y}=\sigma-\inner{\X_x}{\X_y}=\frac{1}{2}d^2\,,
%\eann
%and%, since $\nu_x-kdw\perp X_y$ at the point $(x_0,y_0,t_0)$, we find
\bann
\frac{2}{d^2}\left.\inner{\X_x}{\nu_x-kdw}\right|_{(x_{0},y_{0},t_0)}=\frac{2}{d^2}\left.\inner{\X_x-\X_y}{\nu_x-kdw}\right|_{(x_{0},y_{0},t_0)}=-k(x_0,y_0,t_0)\,.
\eann
Finally, observe that \eqref{eq:Dk} implies
\bann
\frac{2}{d^2}\inner{dw}{\pd^x_i}=R_i{}^p\pd_{x^p}k\,.
\eann
Using these observations, and the vanishing of $\pd_{y^i}k$, we obtain
\ba
\hspace{-4pt}\left(\pd_t-{\widehat\eL}\,\right)k={}& \lb|\xA|^2_F-\sigma \tr(\dot F_x)\rb k+2\sigma F_x+2\dot F^{ij}_x\pd_{x^i}k\,R_j{}^p\pd_{x^p}k+\frac{2}{d^2} \Big\{\yF-\xF\nonumber\\
&+\dot F^{ij}_x\big[(k\delta_{ij}-\xA_{ij})-2{\Lambda_i}^p(k\delta_{pj}-\xA_{pj}) +{\Lambda_i}^p{\Lambda_j}^q(k\delta_{pq}-\yA_{pq})\big]\Big\}\label{eq:a}
\ea
at any off-diagonal extremum $(x_0,y_0,t_0)$, where we have defined $R:=(\W^x-kI)^{-1}$ with $I$ denoting the identity.

Applying the inequalities \eqref{eq:lowersupport}, we obtain
\bann
0\geq{}& (\pd_t-\widehat\eL)(k-\phi)\\
\geq{}&-(\pd_t-\eL)\phi+\lb |\W^x|_F^2-\sigma \tr(\dot F_x)\rb k+2\sigma F_x+2\dot F^{ij}_x\pd_ik\,R_j{}^p\pd_pk\\
&+ \frac{2}{d^2}\Big\{F_y-F_x+\dot F^{ij}_x\Big[(k\delta_{ij}-\xW_{ij}) -2{\Lambda_i}^p(k\delta_{pj}-\xW_{pj})+{\Lambda_i}^p{\Lambda_j}^q(k\delta_{pq}-\yW_{pq})\Big]\Big\}\,.
\eann
It remains to demonstrate non-negativity of the term on the second line for some choice of the matrix $\Lambda$. Since we are free to choose the orthonormal basis at $y_0$ such that $\W$ is diagonalised, this follows from the following proposition.

\begin{prop}\label{prop:interiorestimate}
Let $f:\Gamma_+\to \R$ be a smooth, symmetric function which is monotone increasing in each variable and inverse-concave, and let $F:\mathcal{C}_+\to \R$ be the function defined on the cone $\mathcal{C}_+$ of positive definite symmetric matrices by $F(A)=f(\lambda(A))$, where $\lambda$ denotes the eigenvalue map. Then for any $k\in \R$, any diagonal $B\in \mathcal{C}_+$, and any $A\in \mathcal{C}_+$ with $k<\min_i\{\lambda_i(A)\}$, we have
\bann
0\leq F(B)-F(A)+\dot F^{ij}(A)\sup_{\Lambda}\Big[(k\delta_{ij}-A_{ij}) -2{\Lambda_i}^p(k\delta_{pj}-A_{pj})+{\Lambda_i}^p{\Lambda_j}^q(k\delta_{pq}-B_{pq})\Big]\,.
\eann
\end{prop}
\begin{proof}[Proof of Proposition \ref{prop:interiorestimate}]
Since the expression in the square brackets is quadratic in $\Lambda$, it is easy to see that the supremum is attained with the choice $\Lambda=(A-kI)\cdot(B-kI)^{-1}$, where $I$ denotes the identity matrix. %Making this choice, we obtain
%\bann
%\left(\pd_t-{\widehat\eL}\,\right)k={}& |\xA|^2_F k + 2\dot F^{ij}_x\pd_ik{{(\W^x-kI)^{-1}}_j}^p\pd_pk-\sigma \tr(\dot F_x)\widehat k+2\sigma F_x\\
%&+\frac{2}{d^2}\Big\{\yF-\xF+\dot F^{ij}_x\big[(k\delta_{ij}-\xA_{ij})-2{\Lambda_i}^p(k\delta_{pj}-\xA_{pj}) +{\Lambda_i}^p{\Lambda_j}^q(k\delta_{pq}-\yA_{pq})\big]\Big\}\,.
%\eann
Thus, given any $A\in \mathcal{C}_+$, we need to show that
\bann
0\leq Q_A(B):=F(B)-F(A)-\dot F^{ij}(A)\lb{(A-kI)}_{ij}-\lsb(A-kI)\cdot(B-kI)^{-1}\cdot(A-kI)\rsb_{ij}\rb\,.
\eann
Since $B$ is diagonal, and the expression $Q_A(B)$ is invariant under similarity transformations with respect to $A$, we may diagonalise $A$ to obtain
\bann
Q_A(B):={}&f(b)- f(a)-\dot f^{i}(a)\left[(a_i-k)-\frac{(a_i-k)^2}{b_i-k}\right]\,,
\eann
where we have set $a=\lambda(A)$ and $b=\lambda(B)$. We are led to consider the function $q_a$ defined on $\Gamma_+$ by 
\bann
q_a(z):=f(z)-f(a)-\dot f^i(a)\left[(a_i-k)-\frac{(a_i-k)^2}{z_i-k}\right]\,.
\eann
We compute
\bann
\dot q_a^i={}&\dot f^i-\dot f^{i}(a)\frac{(a_i-k)^2}{(z_i-k)^2}\,,
\eann
and
\bann
\ddot q_a^{ij}={}&\ddot f^{ij}+2\dot f^{i}(a)\frac{(a_i-k)^2}{(z_i-k)^3}\delta^{ij}=\ddot f^{ij}+2\frac{\dot f^i\delta^{ij}}{z_i-k}-2\frac{\dot q_a^i\delta^{ij}}{z_i-k}\,.
\eann
It follows that
\ba\label{eq:inverse-concave1}
\ddot q_a^{ij}+2\frac{\dot q_a^i\delta^{ij}}{z_i-k}={}&\ddot f^{ij}+2\frac{\dot f^i\delta^{ij}}{z_i-k}> \ddot f^{ij}+2\frac{\dot f^i\delta^{ij}}{z_i}\geq 0\,,
\ea
where the last inequality follows from inverse-concavity of $f$ \cite[Corollary 5.4]{An07}. %Therefore,
%\bann
%\ddot q^{ij}+2\frac{\dot q^i\delta^{ij}}{z_i-k}>0\,.
%\eann
%In particular, $q$ has a unique local minimum. Since $q\to\infty$ as $z\to \pd\Gamma_k$, the minimum of $q$ is attained at the point $z=\kappa^x$ (where $\dot q=0$). But $q(\kappa^x)=0$. It follows that $Q\geq 0$.
Thus the minimum of $q$ is attained at the point $z=a$, where it vanishes. This completes the proof.
\end{proof}
This completes the proof in the interior case.\\

\noindent\textbf{The boundary case} (Cf. \cite[Theorem 3.2]{An07}).\\

We now consider the case that $\inf_{\M}k(x_0,\cdot,t_0)$ occurs on the boundary of $\widehat M$; that is, $\underline k(x_{0},t_{0})=\A_{(x_0,t_0)}(y_0,y_0)$ for some $y_{0}\in S_{(x_{0},t_{0})}\M$. Consider the function $K$ defined on $T\M\times[0,T)$ by $K(x,y,t)=\A_{(x,t)}(y,y)$. %Now define a smooth unit vector field $\xi$ (on a possibly smaller backwards parabolic neighbourhood $P'$ of $(x_0,t_0)$) by choosing $\xi(x_{0},t_{0})=\xi_{0}$, and extending in the spatial directions by parallel translation along geodesics, and extending in the time direction by solving
%\begin{equation}\label{eq:xi}
%\begin{dcases} 
%\;\frac{\partial\xi^i}{\partial t}=F\xi^j{\W_j}^i\,;& t<t_0\\
%\hspace{11pt}\xi^i=\xi^i_0\,;& t=t_0\,.
%\end{dcases}
%\end{equation}
Then the function $\Phi(x,y,t):=\phi(x,t)g_{(x,t)}(y,y)$ is a lower support for $K$ at $(x_0,y_0,t_0)$. %That is, $K(x,y,t)\geq \Phi(x,y,t)$ with equality at $(x_0,y_0,t_0)$. 
In particular, 
\ba\label{eq:Lowersupport}
\begin{split}
\pd_t(K-\Phi)\leq{}& 0\\
\mbox{and}\quad \widehat \eL (K-\Phi)\geq{}& 0
\end{split}
\ea 
at $(x_0,y_0,t_0)$ for any elliptic operator $\widehat\eL$ on $T\M$. We require the operator project to $\eL$ on the first factor (at least at the point $(x_0,y_0,t_0)$), which leads us to consider an operator $\widehat \eL$ locally of the form $\widehat\eL=\dot F^{ij}_x({\pd_x^i-{\Lambda_i}^p\pd_{y^p}})({\pd_x^j-{\Lambda_j}^q\pd_{y^q}})$, where $\{x^i,y^i\}_{i=1}^n$ are co\"ordinates for $TM$ near $(x_0,y_0)$. We choose these co\"ordinates such that $\{x^i\}_{i=1}^n$ are normal co\"ordinates on $M$ (with respect to $g_{t_0}$) based at $x_0$, and $\{y^i\}_{i=1}^n$ are the corresponding fibre co\"ordinates (defined by $(x,y)=(x,y^i\pd_{x^i})$ for tangent vectors $(x,y)$ near $(x_0,y_0)$). Moreover, we may assume that $\{\pd_{x^i}|_{x_0}\}_{i=1}^n$ is a basis of eigenvectors of $\W_{(x_0,t_0)}$ with $y_0=\pd_{x^1}|_{x_0}$. 

Writing locally $K-\Phi=y^ky^l(\A_{kl}-g_{kl})$, we find
\bann
(\pd_{x^i}-{\Lambda_i}^p\pd_{y^p})(K-\Phi)=y^ky^l\lb\pd_{x^i}\A_{kl}-\pd_{x^i}\phi\, g_{kl}\rb-2{\Lambda_i}^py^k\lb\A_{kp}-\phi\, g_{kp}\rb\,.
\eann
Thus, at the point $(x_0,y_0,t_0)$, we obtain
\bann
0=(\pd_{x^i}-{\Lambda_i}^p\pd_{y^p})(K-\Phi)=\cd_{i}\A_{11}-\cd_{i}\phi\,.
\eann
We next compute
\bann
(\pd_{x^i}-{\Lambda_i}^p\pd_{y^p})(\pd_{x^j}-{\Lambda_j}^q\pd_{y^q})(K-\Phi)={}&y^ky^l\big(\pd_{x^i}\pd_{x^j}\A_{kl} -\pd_{x^i}\pd_{x^j}\phi\,g_{kl}-\pd_{x^i}\phi\,\pd_{x^j}g_{kl}\\
{}&-\pd_{x^j}\phi\,\pd_{x^i}g_{kl}-\phi\,\pd_{x^i}\pd_{x^j}g_{kl}\big)\\
{}&-2{\Lambda_j}^qy^k\lb\pd_{x^i}\A_{kq}-\pd_{x^i}\phi\,g_{kq}-\phi\,\pd_{x^i}g_{kp}\rb\\
{}&-2{\Lambda_i}^py^k\lb\pd_{x^j}\A_{kp}-\pd_{x^j}\phi\,g_{kp}-\phi\,\pd_{x^j}g_{kp}\rb\\
{}&+2{\Lambda_i}^p{\Lambda_j}^q\lb\A_{pq}-\phi\,g_{pq}\rb\,.
\eann
At the point $(x_0,y_0,t_0)$, we obtain
\ba\label{eq:Phixx}
\widehat\eL(K-\Phi)={}&\eL\A_{11}-\eL\phi\nonumber\\
{}&-2\dot F^{ij}\big[2{\Lambda_i}^p\lb\cd_j\A_{1p}-\cd_j\A_{11}\delta_{1p}\rb -{\Lambda_i}^p{\Lambda_j}^q(\A_{pq}-\A_{11}\delta_{pq})\big]\,.
\ea
Finally, we compute the time derivative
\bann
\pd_t(K-\Phi)=y^ky^l\big(\pd_t\A_{kl}-\pd_t\phi\,g_{kl}-\phi\,\pd_tg_{kl}\big)\,,
\eann
which at $(x_0,y_0,t_0)$ becomes
\ba\label{eq:Phit}
\pd_t(K-\Phi)=\pd_t\A_{11}-\A_{11}\pd_tg_{11}-\pd_t\phi\,.
\ea
Let us recall the evolution equations for $\A$ and $g$ \cite{An94,An94b}:
\ba
\pd_t\A_{ij}={}&\eL\A_{ij}+\ddot F^{pq,rs}\cd_i\A_{pq}\cd_j\A_{rs}-2\A^2_{ij}F+\lb\A_{ij}-\sigma\tr(\dot F)\rb|\W|^2_F+2\sigma\A_{ij}\,,\label{eq:evh}
\ea
and
\ba
\pd_tg_{ij}={}& -2F\A_{ij}\,.\label{eq:evg}
\ea
Putting \eqref{eq:Phixx} and \eqref{eq:Phit} together, and applying the evolution equations \eqref{eq:evh} and \eqref{eq:evg}, and the inequalities \eqref{eq:Lowersupport}, we obtain
\ba\label{eq:lastterm}
0\geq (\pd_t-\eL)(K-\Phi)={}&-(\pd_t-\eL)\phi+(|\W|_F^2-\sigma\tr(\dot F))\phi+2\sigma F+\ddot F^{pq,rs}\cd_1\A_{pq}\cd_1\A_{rs}\nonumber\\
{}&+2\dot F^{ij}{\Lambda_i}^p\Big[2\lb\cd_j\A_{1p}-\cd_j\A_{11}\delta_{1p}\rb -{\Lambda_j}^q(\A_{pq}-\A_{11}\delta_{pq})\Big]
\ea
at the point $(x_0,y_0,t_0)$. Note that the term in the last line with $p=1$ vanishes. 

Using a trick of Brendle \cite[Proposition 8]{Br12} (see also \cite[Theorem 7]{AnLi12}) we also obtain $\cd_1\W_{11}=0$ at the point $(x_0,t_0)$:

\begin{lem}\label{lem:gradW} 
$\cd_1\W_{11}$ vanishes at $(x_0,y_0,t_0)$.
\end{lem}
\begin{proof}[Sketch proof of Lemma \ref{lem:gradW}] 
Since $\kappa_1(x_0,t_0)=\inf_{y\neq x_0} k(x_0,y,t_0)$, we have
\bann
0\leq Z(x_0,y,t_0):= 2\inner{\X(x_0,t_0)-\X(y,t_0)}{\nu(x_0,t_0)}-\kappa_1(x_0,t_0)\norm{\X(x_0,t_0)-\X(y,t_0)}^2
\eann
for all $y\in \M$. In particular, $0\leq f(s):=Z(x_0,\gamma(s),t_0)$ for all $s$, where $\gamma(s):=\exp_{x_0}sy_0$. It is straightforward to compute $0=f(0)=f'(0)=f''(0)$, which, since $f\geq 0$, implies that $f'''(0)=0$. %(assuming otherwise, its easy to show that $f$ becomes negative for $s$ arbitrarily close to 0). 
But a further straightforward computation yields $f'''(0)=2\cd_{1}\A_{11}$. %We can choose an orthonormal basis $\{e_i\}_{i=1}^n$ at $x_0$ such that $e_1=\xi_0$ and $\A$ is diagonalised. Extending this basis to a neighbourhood of $x_0$ by parallel transport, we obtain (at $x_0$) $0=e_i\A_{jk}=\cd_i\A_{jk}$ whenever $j\neq k$. In particular, $0=\cd_1\A_{jk}$ whenever $j\neq k$ and $0=\cd_l\A_{l1}=\cd_1\A_{ll}$ whenever $l\neq 1$. We conclude that $\cd_1\A=0$ at $x_0$. In particular, $\ddot F (\cd_{\xi_0}\A,\cd_{\xi_0}\A)=0$. 
\end{proof}

Applying the following proposition to \eqref{eq:lastterm} completes the proof.

\begin{prop}\label{prop:boundaryestimate}
Let $f:\Gamma_+\to \R$ be a smooth, symmetric function which is monotone increasing in each variable and inverse-concave, and let $F:\mathcal{C}_+\to \R$ be the function defined on the cone $\mathcal{C}_+$ of positive definite symmetric matrices by $F(A)=f(\lambda(A))$, where $\lambda$ denotes the eigenvalue map. If $A\in \mathcal{C}_+$ and $y$ is an eigenvector of $A$ corresponding to its smallest eigenvalue, then, for any totally symmetric 3-tensor $T$ with $T(y,y,y)=0$, we have
\bann
%2\dot F^{kl}&\lsb(A-A(v))^{-1}\rsb_l{}^pT(y,y,e_k)T(y,y,e_p)\\
0 \leq{}& y^iy^j\ddot F^{pq,rs}T_{ipq}T_{jrs}+2\dot F^{kl}\sup_{\Lambda}\Big[2{\Lambda_k}^py^i\lb T_{ilp}-y^ry^sT_{lrs}\delta_{iq}\rb-{\Lambda_k}^p{\Lambda_l}^q(A_{pq}-y^ry^sA_{rs}\delta_{pq})\Big]
\eann
at the matrix $A$. Moreover, equality holds only if $T(v,y,y)=0$ for all $v\in \R^n$.
\end{prop}
\begin{proof}[Proof of Proposition \ref{prop:boundaryestimate}]
We first observe that it suffices to prove the claim for those $A\in\mathcal{C}_+$ having distinct eigenvalues: The expression
$$
Q:=2\dot F^{kl}\Big[2{\Lambda_k}^py^i\lb T_{ilp}-y^ry^sT_{lrs}\delta_{iq}\rb-{\Lambda_k}^p{\Lambda_l}^q(A_{pq}-y^ry^sA_{rs}\delta_{pq})\Big]
$$
is continuous in $A$, and hence the supremum over $\Lambda$ is upper semi-continuous in $A$; so the general case follows by taking a sequence of matrices $A^{(k)}\in \mathcal{C}_+$ approaching $A$ with each $A^{(k)}$ having distinct eigenvalues.

So suppose that $A$ has distinct eigenvalues and let $\{e_i\}_{i=1}^n$ be an orthonormal frame of eigenvectors of $A$ with $e_1=y$. Then %the expression we wish to optimise becomes
$$
Q=2\dot F^{kl}\Big[2{\Lambda_k}^p\lb T_{1lp}-T_{l11}\delta_{1q}\rb-{\Lambda_k}^p{\Lambda_l}^q(A_{pq}-A_{11}\delta_{pq})\Big]\,.
$$
Observe that the supremum over $\Lambda$ occurs when $\Lambda_{lq}=(\lambda_q-\lambda_1)^{-1}T_{1lq}$ for $i,p>1$. With this choice, we obtain
$$
Q=2\dot F^{kl}R^{pq}T_{1kq}T_{1lp}\,,
$$
where $R^{pq}:=(\lambda_p-\lambda_1)^{-1}\delta^{pq}$ for $p,q\neq 1$ and zero otherwise. Therefore, it suffices to prove that
\bann
%2\dot F^{kl}R_k{}^pB_{1l}B_{1p} 
0\leq \lb\ddot F^{pq,rs}+2\dot F^{pr}R^{qs}\rb B_{pq}B_{rs}
\eann
for any symmetric $B$ with $B_{11}=0$ with equality only if $B_{1q}=0$ for all $q$. The expression we want to estimate may be written in terms of the function $f$ as follows (see, for example, \cite[Theorem 5.1]{An07}):
\bann
\lb\ddot F^{pq,rs}+2\dot F^{pr}R^{qs}\rb B_{pq}B_{rs}={}&\ddot f^{pq}B_{pp}B_{qq}+\sum_{p\neq q}\frac{\dot f^p-\dot f^q}{\lambda_p-\lambda_q}B_{pq}^2+2\sum_{p=1,\,q=2}^n\frac{\dot f^p}{\lambda_q-\lambda_1}B_{pq}^2\\
={}&\ddot f^{pq}B_{pp}B_{qq}+2\sum_{p>1,\,q>1}\frac{\dot f^p\delta^{pq}}{\lambda_p-\lambda_1} B_{pp}B_{qq}+\sum_{p\neq q}\frac{\dot f^p-\dot f^q}{\lambda_p-\lambda_q}B_{pq}^2\\
{}&+2\sum_{p=2}^n\frac{\dot f^1}{\lambda_p-\lambda_1}B_{p1}^2+2\sum_{p>1,\,q>1,\atop p\neq q}\frac{\dot f^p}{\lambda_q-\lambda_1}B_{pq}^2\,.
\eann
We first estimate
\bann
\ddot f^{pq}B_{pp}B_{qq}+2\sum_{p>1,q>1}\frac{\dot f^p\delta^{pq}}{\lambda_p-\lambda_1}B_{pp}B_{qq}\geq{}&\ddot f^{pq}B_{pp}B_{qq}+2\sum_{p=2,q=2}^n\frac{\dot f^p}{\lambda_p}\,\delta^{pq}B_{pp}B_{qq}\\
={}&\lb\ddot f^{pq}+2\frac{\dot f^p}{\lambda_p}\,\delta^{pq}\rb B_{pp}B_{qq}\geq 0\,,
\eann
where the final inequality follows from inverse-concavity of $f$ \cite[Theorem 2.1]{An07}. The remaining terms are
\bann
\sum_{p\neq q}\frac{\dot f^p-\dot f^q}{\lambda_p-\lambda_q}B_{pq}^2&+2\sum_{p=2}^n\frac{\dot f^1}{\lambda_p-\lambda_1}B_{p1}^2+2\sum_{p>1,\, q>1,\atop p\neq q}\frac{\dot f^p}{\lambda_q-\lambda_1}B_{pq}^2\\
={}& \sum_{p>1,\, q>1,\atop p\neq q}\lb\frac{\dot f^p-\dot f^q}{\lambda_p-\lambda_q}+2\frac{\dot f^p}{\lambda_q-\lambda_1}\rb B_{pq}^2+2\sum_{p=2}^n\lb\frac{\dot f^p-\dot f^1}{\lambda_p-\lambda_1}+\frac{\dot f^1}{\lambda_p-\lambda_1}\rb B_{p1}^2\\
\geq{}& \sum_{p>1,\, q>1,\atop p\neq q}\lb\frac{\dot f^p-\dot f^q}{\lambda_p-\lambda_q}+\frac{\dot f^p}{\lambda_q}+\frac{\dot f^q}{\lambda_p}\rb B_{pq}^2+2\sum_{p=2}^n\lb\frac{\dot f^p}{\lambda_p-\lambda_1}\rb B_{p1}^2\,.
\eann
The first term is non-negative by inverse-concavity of $f$ \cite[Corollary 5.4]{An07} and the second term is clearly non-negative and vanishes only if $B_{1q}=0$ for all $q>1$. This completes the proof.
\end{proof}

This completes the proof that $\underline k$ is a viscosity supersolution of \eqref{eq:LF}.
\end{proof}

Since the speed function satisfies \eqref{eq:LF}, the statement of Theorem \ref{thm:NC} follows from a simple comparison argument for viscosity solutions (cf. \cite{NC,NC2}):

Define $\varphi(t):= \E^{2\sigma \eta t}\lb\inf_{M\times\{t\}}\underline k/F-1/\eta\rb$, where $\eta>0$ is such that $\sigma\eta\geq \sigma\tr(\dot F)$; that is, $\eta\geq\tr(\dot F)$ for flows in the sphere, $\eta\leq \tr(\dot F)$ for flows in hyperbolic space, and $\eta> 0$ for flows in Euclidean space. We claim that $\varphi$ is non-decreasing. It suffices to prove that $\underline k(\,\cdot\,,t)-\lb 1/\eta+\E^{-2\sigma \eta t}\varphi(t_0)-\varepsilon \E^{L(t-t_0)}\rb F(\,\cdot\,,t)> 0$ for any $t_0\in [0,T), t\in[t_0,T)$ and any $\varepsilon>0$, where we have set $L:=1-2\sigma\eta$. Taking $\varepsilon \to 0$ then gives $\underline k(\,\cdot\,,t)-\lb 1/\eta + \E^{-2\sigma \eta t}\varphi(t_0)\rb F(\,\cdot\,,t)\geq 0$; that is, $\varphi(t)\geq \varphi(t_0)$ for all $t \geq t_0$ for any $t_0$. Now, at time $t_0$ we have $\underline k(x,t_0)-\lb 1/\eta+\E^{-2\sigma \eta t_0}\varphi(t_0)-\varepsilon\rb F(x,t_0)\geq \varepsilon F(x,t_0)>0$. So suppose, contrary to the claim, that there is a point $(x_1,t_1)\in \M\times [t_0,T)$ and some $\varepsilon>0$ such that $\underline k(x_1,t_1)-\lb 1/\eta+\E^{-2\sigma\eta t_1}\varphi(t_0)-\varepsilon \E^{t_1-t_0}\rb F(x_1,t_1)=0$. Assuming that $t_1$ is the first such time, this means precisely that the function $\phi(x,t):=\lb1/\eta+\E^{-2\sigma\eta t}\varphi(t_0)-\varepsilon \E^{L(t-t_0)}\rb F(x,t)$ is a lower support for $\underline k$ at $(x_1,t_1)$. But $\underline k$ is a viscosity supersolution of \eqref{eq:LF}, so that, at the point $(x_1,t_1)$, $\phi$ satisfies
\bann
0\geq{}& -\lb\pd_t-\eL\rb\phi+\lb|\W|^2_F-\sigma \tr(\dot F)\rb\phi+2\sigma F\\
={}&\lb 2\eta \sigma\E^{-2\sigma \eta t_1}\varphi(t_0)+L\varepsilon \E^{L(t_1-t_0)}\rb F - \lb \frac{1}{\eta }+\E^{-2\sigma \eta t_1}\varphi(t_0)-\varepsilon \E^{L(t_1-t_0)}\rb \lb|\W|^2_F+\sigma \tr(\dot F)\rb\\
&+\lb \frac{1}{\eta }+\E^{-2\sigma \eta t_1}\varphi(t_0)-\varepsilon \E^{L(t_1-t_0)} \rb F\lb|\W|^2_F-\sigma \tr(\dot F)\rb+2\sigma F\\
%={}&\lb 2\eta \sigma\E^{2\sigma \eta t_1}\varphi(t_0)+L\varepsilon \E^{L(t_1-t_0)}\rb F-2\sigma\lb \frac{1}{\eta }+\E^{2\sigma \eta t_1}\varphi(t_0)+\varepsilon \E^{L(t_1-t_0)} \rb F\tr(\dot F)+2\sigma F\\
={}& 2\E^{-2\sigma \eta t_1}\varphi(t_0)F\lb\sigma\eta-\sigma\tr(\dot F)\rb+\varepsilon \E^{L(t_1-t_0)}F\lb L+2\sigma \tr(\dot F)\rb+\frac{2}{\eta}F\lb\sigma\eta-\sigma\tr(\dot F)\rb\\
\geq{}&\varepsilon\E^{L(t_1-t_0)}>0\,,
\eann
where we used $L:=1-2\sigma\eta$, and $\sigma\eta\geq \sigma \tr(\dot F)$ in the last line. This contradiction proves that $\phi$ could not have reached zero on $[t_0,T)$, which, as explained above, proves that $\varphi$ is non-decreasing. Therefore,
\bann
\lb\frac{\underline k(x,t)}{F(x,t)}-\frac{1}{\eta}\rb \E^{2\sigma \eta t}\geq \inf_{M\times\{t\}}\lb\frac{\underline k}{F}-\frac{1}{\eta}\rb \E^{2\sigma \eta t}=:\varphi(t)\geq\varphi(0)=\inf_{M\times\{0\}}\lb\frac{\underline k}{F}-\frac{1}{\eta}\rb\,.
\eann

This proves Theorem \ref{thm:NC}.

\section{Convex solutions in Euclidean space}

We now give an application of non-collapsing to flows of convex hypersurfaces; namely, we give a new proof that convex hypersurfaces contract to round points under the flow \eqref{eq:CF} in Euclidean space when the speed is both concave and inverse-concave \cite{An07}. % using a pinching estimate and a geometric lemma relating the width ratio to the pinching ratio. The particular case $F=H$ was obtained by Huisken \cite{Hu84} by showing that the pinching ratio of the solution improves at a singularity pinching estimate. Later, Hamilton \cite{Ha92} gave a very different proof  by applying the strong maximum principle By contrast, 
%We note that our proof only requires the strong maximum principle (applied to Proposition \ref{prop:visc}) and the well-known fact that the only compact, convex hypersurface of $\R^{n+1}$ with $F(\W)$ constant are spheres \cite{EckHui}. %The exposition will remain brief as the only new ingredient is the use of two-sided non-collapsing. The rest of the details appear in the references \cite{Hu84,An94,An94b,AC}.

We begin with some background results on fully non-linear curvature flows \eqref{eq:CF}. Given smooth, compact initial data on which $F$ is defined, we obtain unique solutions for a short time \cite{An94}. Since we can enclose the initial hypersurface by a large sphere, which shrinks to a point in finite time, the avoidance principle (see, for example, \cite[Theorem 5]{NC}) implies that the maximal time $T$ of existence of the solution must be finite. For inverse-concave speeds, the non-collapsing estimate yields a preserved cone of curvatures for the flow, since $\kappa_{\min}\geq \underline k\geq k_0 F$. This implies that the flow is uniformly parabolic, since positive bounds for $\dot F$ on the intersection of the preserved cone with the unit sphere $\{|\W|=1\}$ extend to bounds on the entire cone. If $F$ is also a concave function, then global regularity of solutions may be obtained by appealing to the scalar parabolic theory of Krylov-Safanov \cite{KS} and Evans and Krylov \cite{E,K} (cf. \cite{surfaces,AM}). We conclude that the solution will remain smooth until $\max_{\M\times \{t\}}F\to\infty$ as $t\to T$. 

The key to our proof of the convergence theorem is showing that the normalised interior and exterior ball curvatures improve to unity at a singularity. This is achieved using a blow-up argument and applying the strong maximum principle.

\begin{thm}\label{thm:improvingNC}
Suppose $F$ is concave and inverse-concave. Then along any convex solution $X:\M^n\times [0,T)\to\R^{n+1}$ of \eqref{eq:CF} the following estimates hold:
\ben
\item For every $\varepsilon>0$ there exists $F_\varepsilon<\infty$ such that
\bann
F>\F_\varepsilon\quad\Rightarrow\quad \overline k\leq (1+\varepsilon)F\,.
\eann
\item For every $\delta>0$ there exists $F_\delta<\infty$ such that
\bann
F>\F_\delta\quad\Rightarrow\quad \underline k\geq (1-\delta)F\,.
\eann
\een
\end{thm}
\begin{proof}
We will blow the solution up at a point where $F$ is becoming large. Applying the strong maximum principle, and making use of the gradient term appearing in \eqref{eq:LF}, we find that this limit must be a shrinking sphere, from which the claims follow. We note that the only auxillary result we require is the fact that the only closed, convex hypersurfaces of $\R^n$ with $F$ constant are spheres \cite{EckHui}. When $F$ is the mean curvature, this is a well-known theorem of Alexandrov \cite{Al}.

Suppose the first estimate were false. Then there exists a sequence $(x_i,t_i)\in M\times[0,T)$ such that $F(x_i,t_i)\to\infty$ but $\frac{\overline k}{F}(x_i,t_i)\to (1+\varepsilon_0)$, where $\varepsilon_0>0$. By Theorem \ref{thm:NC}, $\varepsilon_0<\infty$. Set $\lambda_i:=F(x_i,t_i)$ and consider the blow-up sequence
\bann
\X_i(x,t):={}&\lambda_i\lb \X\left(x,\lambda_i^{-2}t+t_i\right)-\X_i(x_i,t_i)\rb\,.
\eann
It is easily checked that $\X_i:M^n\times \big[-\lambda_i^2t_i,0]\to \R^{n+1}$ is a solution of the flow \eqref{eq:CF} for each $i$. Moreover, for each $i$, we have $\max_{M\times [-\lambda_i^2t_i,0]}F_i=F_i(x_i,0)=1$ and $\X_i(x_i,0)=0$. It follows that the sequence $\X_i$ converges locally uniformly along a subsequence to a smooth limit flow $\X_\infty:\M_\infty\times (-\infty,0]\to\R^{n+1}$ (cf. \cite{Baker,Breuning,Cooper}).

Since the ratio $\underline k/ F$ is invariant under rescaling, we have 
$$
\frac{\underline k_i}{F_i}(x_i,0)=\frac{\underline k}{F}(x_i,t_i)\geq k_0>0\,,
$$
which implies that the image of each $X_i$ is contained in a compact set. It follows that the convergence is global, so that $\M_\infty\cong \M$. 

We now show that $\underline k/F$ must be constant on the limit flow $\X_\infty$; for if not, by Proposition \ref{prop:visc} and the strong maximum principle (see, for example, \cite{SMP}), its spatial maximum must must decrease monotonically, by an amount $L$ say, on some sub-interval $[t_1,t_2]$ of $(-\infty,0]$. But then (passing to the convergent subsequence) there must exist sequences of times $t_{1,i},t_{2,i}\in \big[-\lambda_i^2t_i,0]$ with $t_{1,i}\to t_1$ and $t_{2,i}\to t_2$ such that
\ba\label{eq:strMP}
L-\varepsilon\leq \max_{M}\frac{\overline k}{F}\lb\,\cdot\,, \lambda_i^{-2}t_{1,i}+t_i\rb-\max_{M}\frac{\overline k}{F}\lb\,\cdot\,, \lambda_i^{-2}t_{2,i}+ t_i\rb
\ea
for any $\varepsilon>0$, so long as $i$ is chosen accordingly large. But since $\lambda_i\to\infty$, the right hand side of \eqref{eq:strMP} converges to zero. It follows that $\sup_{\M\times\{t\}}\overline k/F$ is independent of $t$. Since $\M$ is compact, the space-time supremum of $\overline k/F$ is attained at an interior space-time point, and we deduce that $\overline k/F$ is constant. Since there is a sequence of points $x_i$ for which $\frac{\overline k_i}{F_i}(x_i,0)\to (1+\varepsilon_0)$, we must have $\overline k\equiv (1+\varepsilon_0) F$ on the limit. In particular, we have $0\equiv (\pd_t-\eL)\lb\underline k-(1+\varepsilon_0)F\rb$. But then, computing as in Proposition \ref{prop:visc}, we find $0\equiv \cd \overline k\equiv (1+\varepsilon_0) \cd F$ due to Propositions \ref{prop:interiorestimate} and \ref{prop:boundaryestimate}. But the only closed, convex hypersurfaces of $\R^n$ with $F$ constant are spheres \cite{EckHui}, which satisfy $\overline k\equiv F$. This contradicts $\varepsilon_0>0$.

The proof of the second estimate is similar.
\end{proof}

\begin{rem}
We note that, for flows by convex speed functions, where exterior non-collapsing holds, the proof of the exerior ball estimate goes through. However, for flows by concave speed functions, where interior non-collapsing holds, the proof of the interior ball estimate does not go through without some additional condition (such as a pinching condition) to ensure that the blow-up limit is convex. In fact, due to the examples constructed by Andrews-McCoy-Zheng \cite[\S 5]{AMZ}, one cannot expect such a result to hold in general.
\end{rem}

We now prove the convergence result:

\begin{thm}\label{thm:Euclidean}
Let $\X:\M^n\times [0,T)\to \R^{n+1}$ be a maximal solution of the curvature flow \eqref{eq:CF} such that the speed is a concave, inverse-concave function of the Weingarten map. Then $\X$ converges to a constant $p\in \R^{n+1}$ as $t\to T$, and the rescaled embeddings
\bann
\widetilde\X(x,t):={}&\frac{\X(x,t)-p}{\sqrt{2(T-t)}}
\eann
converge in $C^2$ as $t\to T$ to a limit embedding with image equal to the unit sphere $S^n$.
\end{thm}

\begin{proof}[Proof of Theorem \ref{thm:Euclidean}]
We first apply Theorem \ref{thm:NC} to show that the solution converges uniformly to a point $p\in \R^{n+1}$ in the Hausdorff metric: Observe that $|\X(x_1,t)-\X(x_2,t)|\leq 2r_+(t)$ for every $x_1, x_2\in\M$ and every $t\in[0,T)$, where $r_+(t)$ denotes the circumradius of $\X(\M,t)$ (this is the radius of the smallest ball in $\R^{n+1}$ that contains the hypersurface $\X(\M,t)$). Since $\X$ remains in the compact region enclosed by some initial circumsphere, it suffices to show that $r_+\to 0$ as $t\to T$. But this follows directly from Theorem \ref{thm:NC}: Since $\underline k(x,t)$ is the curvature of the smallest ball which encloses the hypersurface $\X(\M,t)$, and touches it at $\X(x,t)$, we have
$$
\frac{1}{r_+}\geq \max_{\M\times\{t\}}\underline k\geq k_0\max_{\M\times\{t\}}F\,.
$$
But $\max_{\M\times\{t\}}F\to\infty$.

We now deduce Hausdorff convergence of the rescaled hypersurfaces $\widetilde\X(\M,t)$ to the unit sphere: By Theorem \eqref{thm:improvingNC}, for all $\varepsilon>0$ there is a time $t_\varepsilon\in[0,T)$ such that $r_+(t)\leq (1+\varepsilon)r_-(t)$ for all $t\in[t_\varepsilon,T)$, where $r_-(t)$ is the in-radius of $\X(\M,t)$ (the radius of the largest ball enclosed by $\X(\M,t)$). By the avoidance principle the remaining time of existence at each time $t$ is no less than the lifespan of the shrinking sphere of initial radius $r_-(t)$, and no greater than the lifespan of the shrinking sphere of initial radius $r_+(t)$. These observations yield
\ba\label{eq:Hausdorff}
r_-(t)\leq \sqrt{2(T-t)}\leq r_+(t)\leq (1+\varepsilon)r_-(t)
\ea
for all $t\in[t_\varepsilon,T)$. It follows that the circum- and in-radii of the rescaled solution approach unity as $t\to T$. We can also control the distance from the final point $p$ to the centre $p_t$ of any in-sphere of $\X(M,t)$: For any $t'\in [t,T)$, the final point $p$ is enclosed by $\X(M,t')$, which is enclosed by the sphere of radius $\sqrt{r_+(t)^2-2(t'-t)}$ about $p_t$. Taking $t'\to T$ and applying \eqref{eq:Hausdorff} gives
\bann
|p-p_t|\leq \sqrt{r_+(t)^2-2(T-t)}\leq \sqrt{(1+\varepsilon)^2\cdot 2(T-t)-2(T-t)}
\eann
for all $t\in [t_\varepsilon,T)$. Thus
\ba\label{eq:Hausdorff2}
\frac{|p-p_t|}{\sqrt{2(T-t)}} \leq \sqrt{(1+\varepsilon)^2-1}\,.
\ea
This yields the desired Hausdorff convergence of $\widetilde X$ to the unit sphere.

Next we obtain bounds on the curvature of the rescaled flow $\widetilde\X$: Non-collapsing and the inequalities $r_-(t)\leq \sqrt{2(T-t)}\leq r_+(t)$ derived above yield
\bann
\frac{1}{\sqrt{2(T-t)}}\leq \frac{1}{r_-(t)}\leq \min_{x\in \M}\overline k(x,t)\leq k_0\min_{x\in\M} F\leq\frac{k_0}{K_0}\min_{x\in\M}\underline k(x,t)\leq\frac{k_0}{K_0}\min_{x\in\M}\kappa_{\min}(x,t)\,,
\eann
and
\bann
\frac{1}{\sqrt{2(T-t)}}\geq \frac{1}{r_+(t)}\geq \max_{x\in \M}\underline k(x,t)\geq K_0\max_{x\in\M} F\geq\frac{k_0}{K_0}\max_{x\in\M}\overline k(x,t)\geq\frac{K_0}{k_0}\max_{x\in\M}\kappa_{\max}(x,t)\,.
\eann
By a well-known result of Hamilton \cite[Lemma 14.2]{Ham}, this also implies convergence of the rescaled metrics, and we obtain the desired $C^2$-convergence.
%One can define a normalized flow as in \cite[\S 9]{Hu84}. The time variable in the normalized flow is $\tau := \log(T-t)$, so we have exponential decay of all derivatives of the second fundamental form for the normalized flow, and the argument in \cite[\S 10]{Hu84} can be applied.
\end{proof}

\begin{rem}
One can obtain $C^\infty$-convergence in the above theorem by a standard bootstrapping procedure \cite{Hu84}. Namely, using the time-dependent curvature bounds, one obtains time-dependent bounds on the derivatives of the Weingarten map (of the underlying solution of the flow) to all orders from the curvature derivative estimates. Unfortunately, the resulting estimates do not quite have the right dependence on the remaining time. The correct dependence can be obtained using the interpolation inequality (cf. \cite[\S 9]{Hu84}). This yields exponential $C^\infty$-convergence of the corresponding solution of the normalised flow equation to the unit sphere (cf. \cite[\S 10]{Hu84}). By construction, this yields $C^\infty$-convergence of the rescaled solution.
\end{rem}

\end{document}